\documentclass[12pt,a4paper]{article}

\usepackage{latexsym}
\usepackage{amsmath}
\usepackage{amssymb}
\usepackage{amsthm}
\usepackage{amscd}
\usepackage{mathrsfs}
\usepackage[all]{xy}
\usepackage{graphicx}
\usepackage{comment}
\usepackage{arydshln}

\usepackage{mathtools}

\mathtoolsset{showonlyrefs=true}

\newtheorem{thm}{Theorem}[section]
\newtheorem{cor}[thm]{Corollary}
\newtheorem{lem}[thm]{Lemma}
\newtheorem{conj}[thm]{Conjecture}
\newtheorem{prop}[thm]{Proposition}

\newtheorem{defn}[thm]{Definition}
\newtheorem{rem}[thm]{Remark}

\DeclareMathOperator{\Spec}{Spec}

\DeclareMathOperator{\Gal}{Gal}
\DeclareMathOperator{\id}{id}

\DeclareMathOperator{\Hom}{Hom}

\DeclareMathOperator{\Rep}{Rep}

\DeclareMathOperator{\Ker}{Ker}

\DeclareMathOperator{\Aut}{Aut} 
 
\DeclareMathOperator{\Ind}{Ind}

\DeclareMathOperator{\Irr}{Irr}

\DeclareMathOperator{\Image}{Im}

\DeclareMathOperator{\Lie}{Lie}

\DeclareMathOperator{\ad}{ad}
\DeclareMathOperator{\Res}{Res}
\DeclareMathOperator{\Ad}{Ad}
\DeclareMathOperator{\Sh}{Sh} 
\DeclareMathOperator{\Coh}{Coh}
\DeclareMathOperator{\Fun}{Fun}
\DeclareMathOperator{\Ch}{Ch}

\DeclareMathOperator{\cInd}{c-Ind}

\DeclareMathOperator{\SL}{SL}
\DeclareMathOperator{\GL}{GL}
\DeclareMathOperator{\PGL}{PGL}

\newcommand{\cf}{\textit{cf.\ }}

\makeatletter

\@addtoreset{equation}{section}
\makeatother

\newcommand{\mb}{\mathbb}
\newcommand{\mc}{\mathcal}
\newcommand{\mr}{\mathrm}
\newcommand{\mf}{\mathfrak}
\newcommand{\mbf}{\mathbf}
\newcommand{\ms}{\mathsf}
\newcommand{\msc}{\mathscr}

\newcommand{\ol}{\overline}
\newcommand{\ul}{\underline}
\newcommand{\wh}{\widehat}
\newcommand{\wt}{\widetilde}

\newcommand{\Gm}{\mathbb{G}_{\mathrm{m}}}
\newcommand{\Ga}{\mathbb{G}_{\mathrm{a}}}

\newcommand{\ab}{\mathrm{ab}}
\newcommand{\sep}{\mathrm{sep}}

\begin{document}
\title{Finite Langlands correspondence}
\author{Naoki Imai} 

\date{}

\maketitle

\begin{abstract}
We construct the Langlands correspondence for connected reductive groups over finite fields, which we call the finite Langlands correspondence. We discuss also its relation with the categorical local Langlands correspondence. 
\end{abstract}

\section{Introduction}\label{sec:intro}

In \cite{IVLpfin}, we formulate 
the Langlands correspondence for connected reductive groups over finite fields, which we call the finite Langlands correspondence, introducing Langlands parameters for connected reductive groups over finite fields. 
In the $\GL_n$-case, this correspondence is essentially due to Macdonald \cite{MacZetafingen}. See also \cite{AubMaccor} for a more relevant description. 
In \cite{ColRedffSL}, the finite Langlands correspondence is constructed for $\SL_n$, and its relation to the local Langlands correspondence is studied in details. 
In this paper we construct the finite Langlands correspondence.

Let $G$ be a reductive group over a finite field $k$. 
Let $\ell$ be a prime different from the characteristic of $k$. We write $\Rep (G(k))$ for the category of finite dimensional representations of $G(k)$ over $\ol{\mb{Q}}_{\ell}$.  
Actually we first show the following theorem on a categorical equivalence. 
\begin{thm}[Theorem \ref{thm:RepGcateq}]
	We fix a Whittaker datum of $G^{\circ}$. 
	Then there is a natural equivalence 
	\[
	\Rep (G(k)) \cong \bigoplus_{\mf{o}} \bigoplus_{\mf{c}}
	\bigoplus_{\beta \in \mf{B}_{\mbf{c},\Omega_{\mbf{c}}}^{\circ}} \Sh^{\mc{G}_{\mbf{c}},\tau_{\beta}}(\mc{G}_{\mbf{c}})^{\Omega_{\mbf{c},\beta}}
	\]
	of abelian categories, where $\mf{o}$ runs thorough the semisimple parameters of $G$, $\mf{c}$ runs thorough the unipotent parameters of $G$ with respect to $\mf{o}$. 
Here $\mf{B}_{\mbf{c},\Omega_{\mbf{c}}}^{\circ}$ and $\mc{G}_{\mbf{c}}$ are a finite set and a finite group determined from $\mf{c}$, 
and $\Sh^{\mc{G}_{\mbf{c}},\tau_{\beta}}(\mc{G}_{\mbf{c}})^{\Omega_{\mbf{c},\beta}}$ is a category of $\mc{G}_{\mbf{c}}$-equivarinat sheaves on $\mc{G}_{\mbf{c}}$ with equivariant structure with respect to an action of a finite group $\Omega_{\mbf{c},\beta}$. 
	\end{thm}
Because of the importance of disconnected groups in Langlands program (\cf \cite{KalLLCdis}), we include disconnected groups in Theorem \ref{thm:RepGcateq} as well. 
A key input is Proposition \ref{prop:unipcatdec}, where we describe the category of unipotent representations of a connected reductive group. 
Then we construct the finite Langlands correspondence in 
Theorem \ref{thm:LCfin} using Theorem \ref{thm:RepGcateq}. 
The categorical equivalence in Proposition \ref{prop:unipcatdec} is crucial even if one is finally interested only in the isomorphism classes (for example, the finite Langlands correspondence) by the following reason: 
A parametrizatin of irreducible representations of $G(k)$ is reduced to that of unipotent representations of a disconnected group involving an endoscopic group. 
To extend the parametrization of unipotent representations from a connected group to a disconnected group canonically, we need functoriality of the categorial equivalence in Proposition \ref{prop:unipcatdec}, which we can not see from a bijection between sets of isomorphism classes. 

Using Theorem \ref{thm:RepGcateq}, we also deduce a parametrizatin of irreducible representations of $G(k)$ in the style of \cite{LusChred} in Theorem \ref{thm:parposdual}. In \cite[(13.2.1)]{LusChred}, the center of $G$ is assumed to be connected. See \cite{LusRepdis}, \cite{DiMiLuspar} and \cite{LusIrrSpin} for the general case.  
Our result is new if the center of $G$ is disconnected in the sense that the parametrization is uniquely rigidified. This is possible thanks to functoriality of the categorial equivalence in Proposition \ref{prop:unipcatdec} as explained above. 

In Section \ref{sec:RelcatLL}, we formulate a conjecture relating the finite Langlands correspondence with the categorical local Langlands correspondence. 

\subsection*{Acknowledgements}
The author would like to thank David Vogan for various discussions and helpful comments on the topic of this paper. 
He also thank Teruhisa Koshikawa and Xinwen Zhu for helpful comments. 
\section{Whittaker datum}

We recall a definition of Whittaker datum. 
Let $p$ be a prime number. 
In this section, 
$F$ is a non-archimedean local field of residue characteristic $p$ or a finite field of characteristic $p$. 
Let $G$ be a connected quasi-split reductive group over $F$. 
Let $\Lambda$ be a ring. 
%We put 
%\[\mb{Z}[\mu_{p^{\infty}}]=\bigcup_{m \geq 0} \mb{Z}[\mu_{p^{m}}], \]
%where $\mb{Z}[\mu_{p^{m}}]$ is the subring of $\ol{\mb{Q}}$ generated by all the $p^{m}$-th roots of unity. 

\begin{defn}
Let $B$ be a Borel subgroup of $G$. 
A smooth character $\psi \colon R_{\mr{u}}(B)(F) \to \mb{Z}[1/p]/\mb{Z}$ is called generic if it factors through $R_{\mr{u}}(B)^{\ab}(F)$ and the stabilizer of $\psi$ in $G(F)$ is $Z(G)(F)R_{\mr{u}}(B)(F)$. 

If $\psi \colon R_{\mr{u}}(B)(F) \to \mb{Z}[1/p]/\mb{Z}$ is a generic character and there is an injective homomorphism $\Image (\psi) \hookrightarrow \Lambda^{\times}$, then the obtained homomorphism  $R_{\mr{u}}(B)(F) \to \Lambda^{\times}$ is also called a generic character. 
\end{defn}

\begin{rem}
If $F$ is non-archimedean local field, any smooth character $R_{\mr{u}}(B)(F) \to \mb{Z}[1/p]/\mb{Z}$ factors through $R_{\mr{u}}(B)^{\ab}(F)$ by \cite[4.1 Theorem]{BuHeOnderunip}. 
This may not be true in general when $\lvert F \rvert \leq 3$ (\cf \cite[Lemma 7]{HowdegSt}). 
\end{rem}

\begin{defn}
	A Whittaker datum of $G$ is a $G(F)$-conjugacy class of $(B,\psi)$, where 
	$B$ is a Borel subgroup of $G$, and $\psi$ is a generic character of $R_{\mr{u}}(B)(F)$. 
\end{defn}

We recall the following well-known fact (\cf \cite[\S 3]{DeRegencusp}):

\begin{prop}\label{prop:whitor}
	The set of Whittaker data of $G$ forms a $G^{\ad}(F)/G(F)$-torsor. 
\end{prop}
\begin{proof}
It suffices to show that the conjugate action of $G^{\ad}(F)$ is transitive on the set of Whittaker data of $G$. Therefore we may assume that $G$ is adjoint. 

Let $B$ be a Borel subgroup. We take a maxmal torus $T$ of $G$ contained in $B$. It suffices to show that the generic characters of $R_{\mr{u}}(B)(F)$ are conjugate under the action of $T(F)$. 
Let $\Delta$ be the set of simple roots of $G_{F^{\sep}}$ with respect to $B_{F^{\sep}}$ and $T_{F^{\sep}}$. 
We take a representative $\alpha_1, \ldots , \alpha_m \in \Delta$ of $\Gamma_F \backslash \Delta$. 
For $1 \leq i \leq m$, let $F_i$ be the finite separable extension of $F$ such that $\Gamma_{F_i}$ is the stabilizer of $\alpha_i$ in $\Gamma_F$. Then we have an isomorphism 
\begin{equation}\label{eq:RuBabdec}
R_{\mr{u}}(B)^{\ab} \cong \prod_{1 \leq i \leq m} \Res_{F_i/F} \Ga. 
\end{equation}
If $\psi$ is a generaic character, then it induces a character of $\prod_{1 \leq i \leq m} \Ga (F_i)$ via \eqref{eq:RuBabdec}, and its restriction to each factor $\Ga (F_i)$ is non-trivial by the definition of a generic character. We note that the action of $\Gm (F_i)$ on $\Ga (F_i)$ induces a transitive action of $\Gm (F_i)$ on the set of non-trivial characters $\Ga (F_i) \to \mb{Z}[1/p]/\mb{Z}$ (\cf \cite[1.7 Proposition]{BHLLCGL2}). Therefore it suffices to show that the morphism 
$T \to \prod_{1 \leq i \leq m} \Res_{F_i/F} \Gm$ 
induced by conjugate action via \eqref{eq:RuBabdec} is an isomorphism. It is enough to show this after the base change to $F^{\sep}$, where it follows from the adjointness of $G$.  
\end{proof}

\begin{cor}\label{cor:whuni}
If $F$ is a finite field and $Z(G)$ is connected, the Whittaker datum of $G$ is unique. 
\end{cor}
\begin{proof}
This follows from Proposition \ref{prop:whitor} because 
$H^1(F,Z(G))$ is trivial by Lang's theorem. 
\end{proof}

\section{Langlands parameters}\label{sec:WD} 

We recall the definition of Langlands parameters for reductive groups over finite fields from \cite{IVLpfin}. 

Let $k$ be a finite field with $q$ elements. 
Let $\ol{k}$ be an algebraic closure of $k$. 
%For a positive inter $m$, we write $k_m$ for the finite extension of $k$ of degree $m$ in $\ol{k}$. 
Let $\sigma_q \in \Gal(\ol{k}/k)$ be the $q$-th power  arithmetic Frobenius element. 
%For an algebraic variety $X$ over $k$,let $\mr{F} \colon X \to X$ be the $q$-th power geometric Frobenius morphism.

\begin{defn}\label{def:finiteweilZ}
	We put
	\begin{equation}\label{e:finiteweil}
		I_k = \varprojlim_{k'} k'^\times
	\end{equation}
	where $k'$ runs through the finite extensions of $k$ in $\ol{k}$ and the transition maps are the norm maps. 
	We define the Weil group of $k$ by
	\begin{equation}
		W_k = I_k \rtimes \langle \sigma_q \rangle
	\end{equation}
	where the conjugation by $\sigma_q$ acts on $I_k$ as $q$-th power.
\end{defn}

Next we recall a definition of a Weil--Deligne group for a finite field.  

\begin{defn}\label{def:finweilDel}
	The Weil--Deligne group of the finite
	field $k$ is the locally pro-algebraic group
	scheme
		\begin{equation}
			\mathit{WD}_k =\Ga \rtimes W_k
		\end{equation}
over $\mb{Q}$, where $(\sigma_q^n,w) \in W_k$ acts on $\Ga$ by the multiplication by $q^n$.
\end{defn}

Let $G$ be a connected reductive algebraic group over
$k$. Let $\Lambda$ be an algebra over $\mb{Q}$. 
We write ${}^L G$ for the L-group of $G$ over $\Lambda$. 

\begin{defn}
Assume that $\Lambda$ is a field of characteristic zero. 
An L-parameter of Weil--Deligne type is a morphism of 
		\begin{equation}
			\varphi \colon \mathit{WD}_k \rightarrow {}^L G
		\end{equation}
group schemes over $\Lambda$ which is compatible with the projections to $\Gal (\ol{k}/k)$. 

We say that the L-parameter $\varphi$ is special
		if $\varphi|_{\Ga (K)}(1)$ is a special unipotent element of
		$\wh{G}^{\varphi (I_k)}$.
		We say that $\varphi$ is Frobenius semisimple if
		$\varphi(\sigma_q)$ is semisimple in ${}^L G$.
\end{defn}

We put
\[
 A_{Z_{\wh{G}}(\varphi(I_k))^{\circ}}(\varphi(\Ga)) =
 \pi_0 (Z_{Z_{\wh{G}}(\varphi(I_k))^{\circ}}(\varphi(\Ga))/Z(Z_{\wh{G}}(\varphi(I_k))^{\circ})). 
\]
We define Lusztig's canonical quotient $\ol{A}_{Z_{\wh{G}}(\varphi(I_k))^{\circ}}(\varphi(\Ga))$
	of $A_{Z_{\wh{G}}(\varphi(I_k))^{\circ}}(\varphi(\Ga))$ as in \cite[13.1]{LusChred} using the isomorphism given by \cite[Lemma 4.2]{IVLpfin}. 
We put $\varphi_0=\varphi|_{\Ga \times I_k}$ and
\begin{equation}
	A(\varphi_0)=\pi_0 \left( Z_{\wh{G}}(\varphi_0)/Z(\wh{G}) \right).
\end{equation}
 Further we put
\begin{equation}
	\ol{A}(\varphi_0)=A(\varphi_0)/\Ker ( A_{Z_{\wh{G}}(\varphi(I_k))^{\circ}}(\varphi(\Ga)) \to \ol{A}_{Z_{\wh{G}}(\varphi(I_k))^{\circ}}(\varphi(\Ga))).
\end{equation}

We put
\begin{equation*}
	\widetilde{Z}(\varphi_0)=\{ (g,\sigma_q^m) \in {}^{L}G \mid
	\Ad ((g,\sigma_q^m))(\varphi_0 (x))=\varphi_0 (\Ad (\sigma_q^m)(x)) \ \textrm{for all $x \in \Ga \times I_k$}
	\}.
\end{equation*}
Further we put
\begin{equation}
	\widetilde{A}(\varphi_0)=\widetilde{Z}(\varphi_0) /
	\Ker (Z_{\wh{G}}(\varphi_0) \to \ol{A}(\varphi_0)).
\end{equation}

We have 
$\varphi (\sigma_q) \in \widetilde{Z}(\varphi_0)$. 
Let $\ol{\varphi (\sigma_q)}$ be the image of 
$\varphi (\sigma_q)$ under the natural projection 
\[
\widetilde{Z}(\varphi_0) \to 
\widetilde{A}(\varphi_0). 
\]
We say that two 
L-parameters $\varphi$ and $\varphi'$ of Weil--Deligne type are equivalent 
if the following condition is satisfied: 
there is $g \in \wh{G}$ such that 
$\Ad (g)(\varphi_0)=\varphi_0'$ and 
$\ol{\varphi (\sigma_q)}$ corresponds to 
$\ol{\varphi' (\sigma_q)}$ under the bijection 
\[
\widetilde{A}(\varphi_0) \cong 
\widetilde{A}(\varphi'_0) 
\]
induced by $\Ad (g)$, where 
$\varphi_0=\varphi|_{\Ga \times I_k}$ and 
$\varphi_0'=\varphi'|_{\Ga \times I_k}$. 
Let $\Phi_{\Lambda}(G)$ be the equivalence classes of 
Frobenius semisimple L-parameters over $\Lambda$ of $G$. 
We write $\Phi_{\Lambda}(G)_{\mr{sp}} \subset \Phi_{\Lambda}(G)$ for 
the equivalence classes of special ones. 

We put
\begin{equation}
	A_{\varphi} = Z_{\ol{A}(\varphi_0)} (\ol{\varphi (\sigma_q)}) . 
\end{equation} 
In the finite Langlands correspondence, the irreducible representation of 
$A_{\varphi}$ will parametrizes the L-packet of $\varphi$. 

%\begin{conj}\label{conj:FLC}
%There is a natural surjection 
%\begin{equation*}		\mc{L}_G \colon \Irr_{\ol{\mb{Q}}_{\ell}}(G(k)) \to 		\Phi_{\ol{\mb{Q}}_{\ell}}(G)_{\mr{sp}}. \end{equation*}
%If we fix a Whittaker datum of $G$, then there is a natural bijection between $\mc{L}_G^{-1}(\varphi)$ and $\Irr_{\ol{\mb{Q}}_{\ell}}(A_{\varphi})$ for $\varphi \in \Phi_{\ol{\mb{Q}}_{\ell}}(G)_{\mr{sp}}$. \end{conj}

\section{Categorical equivalence}

Let $H$ be a connected reductive algebraic group over $\ol{k}$ with Frobenius endomorphism $\mr{F}$ with respect to $k$. 
Let $\mbf{c}$ be a two-sided cell of the Weyl group of $H$ such that $\mr{F}(\mbf{c})=\mbf{c}$. 
Let $\ell$ be a prime number different from $p$.  
We write $\Rep_{\mr{u}}^{\mbf{c}}(H^{\mr{F}})$ 
for the category of representation of $H^{\mr{F}}$ over $\ol{\mb{Q}}_{\ell}$ which are finite direct sums of unipotent representations whose associated two-sided is $\mbf{c}$. 

Let $u_{\mbf{c}}$ be an unipotent element in the unipotent conjugacy class of $\wh{H}$ over $\ol{\mb{Q}}_{\ell}$ corresponding to $\mbf{c}$. 
Let $\sigma$ be the automorphism of $\wh{H}$ given by the Frobenius structure of $H$. 
We define $\mc{G}_{\mbf{c}}$ as the Lusztig's canonical quotient of $\pi_0 (Z_{\wh{H}}(u_{\mbf{c}})/Z(\wh{H}))$. 
This is isomorphic to the finite group attached to $\mbf{c}$ constructed in \cite[Chapter 4]{LusChred} by \cite[(13.1.3)]{LusChred} (\cf \cite{LusFamSpr}). 
We take $h_{\mbf{c}} \in \wh{H}$ such that 
$\Ad (h_{\mbf{c}})(\sigma (u_{\mbf{c}}))=u_{\mbf{c}}$. 
We define $\sigma_{\mbf{c}} \in \Aut (\mc{G}_{\mbf{c}})$ by 
$\Ad (h_{\mbf{c}}) \circ \sigma$.

For a finite group $\mc{G}$ with an automorphsim $\tau$, let $\Sh^{\mc{G},\tau}(\mc{G})$ be the category of 
$\mc{G}$-equivalent sheaves of finite dimensional $\ol{\mb{Q}}_{\ell}$-vector spaces over $\mc{G}$ with respect to the $\tau$-twisted conjugation action of $\mc{G}$ on $\mc{G}$. If $\tau$ is the identity, we simply write 
$\Sh^{\mc{G}}(\mc{G})$ for $\Sh^{\mc{G},\id}(\mc{G})$. 

\begin{prop}\label{prop:unipcatdec}
There is a canonical equivalence 
	\[
	\Rep_{\mr{u}}^{\mbf{c}}(H^{\mr{F}}) \cong 
	\Sh^{\mc{G}_{\mbf{c}},\sigma_{\mbf{c}}}(\mc{G}_{\mbf{c}})
	\]
	of abelian categories. 
\end{prop}
\begin{proof}
	By \cite[Proposition 11.3.8]{DiMiRepLie2nd} and the definition of $\mc{G}_{\mbf{c}}$, both sides do not change even if we replace $H$ by the adjoint quotient. 
	Hence we may assume that $H$ is adjoint simple. 
	By \cite[6.3 (a)]{LusUnicatcent}, we have an equivalence 
	\[
	\Rep_{\mr{u}}^{\mbf{c}}(H^{\mr{F}}) \cong \mc{Z}_{\sigma}(\mc{C}^{\mbf{c}}\mc{B}^2)
	\]
	of abelian categories, 
	where $\mc{C}^{\mbf{c}}\mc{B}^2$ is a monoidal abelian category defined in 
	\cite[1.6, 6.1]{LusUnicatcent}, 
	and $\mc{Z}_{\sigma}$ is the categorical center relative to $\sigma$ defined in \cite[\S 5.5]{BrViQuanHopf}. 
	
	For each left cell $\Gamma$ in $\mbf{c}$, we attach 
	a subgroup $\mc{H}_{\Gamma}$ of $\mc{G}_{\mbf{c}}$ as \cite[3.8 Proposition]{LusLeadHec}. We put 
	$X_{\mbf{c}}=\coprod_{\Gamma} \mc{G}_{\mbf{c}}/\mc{H}_{\Gamma}$, where $\Gamma$ run through the left cells in $\mbf{c}$. 
	
	Let 
	$\Coh^{\mc{G}_{\mbf{c}}} (X_{\mbf{c}})$ 
	be the category of $\mc{G}_{\mbf{c}}$-equivariant sheaves of finite dimensional $\ol{\mb{Q}}_{\ell}$-vector spaces over $X_{\mbf{c}}$. 
	Let $\Fun _{\mc{G}_{\mbf{c}}} (X_{\mbf{c}},X_{\mbf{c}})$ be the monoidal category of all $\ol{\mb{Q}}_{\ell}$-linear functors 
	$\Coh^{\mc{G}_{\mbf{c}}} (X_{\mbf{c}}) \to \Coh^{\mc{G}_{\mbf{c}}} (X_{\mbf{c}})$ (\cf \cite[5.1]{BeOsTenCellWII}). 
	We recall that $\mc{G}_{\mbf{c}} \cong \mb{Z}/2\mb{Z}$ if $\mbf{c}$ is exceptional (\cf \cite[1]{OstTenExcCellW}). 
	Let $\omega \in H^3 (\mc{G}_{\mbf{c}},\{\pm 1\})$ be the unique nontrivial cohomology class if $\mbf{c}$ is exceptional, and the trivial one otherwise. 
	Let 
	$\Fun_{\mc{G}_{\mbf{c}},\omega} (X_{\mbf{c}} , X_{\mbf{c}})$ be the tensor category obtained from $\Fun_{\mc{G}_{\mbf{c}}} (X_{\mbf{c}}, X_{\mbf{c}})$ 
	by twisting
	the associativity isomorphisms by $\omega$. 
	By \cite[Theorem 4]{BFOTenCellWIII}, \cite[5.1 Remark]{BeOsTenCellWII} and \cite[Remark 2.20]{OstTenExcCellW}, we have an equivalence 
	\begin{equation*}
		\mc{C}^{\mbf{c}}\mc{B}^2 \cong \Fun_{\mc{G}_{\mbf{c}},\omega} (X_{\mbf{c}}, X_{\mbf{c}}) 
	\end{equation*}
	of tensor categories (\cf Proof of \cite[Corollary 5.4]{BFOChDmod}). 
	
We take a minimum positive integer $n$ such that $\sigma_{\mbf{c}}^n$ acts trivially on $\mc{G}_{\mbf{c}}$. 
We note that $n=1$ when $\mbf{c}$ is exceptional. 
%We put $\mc{G}_{\mbf{c},F}=\mc{G}_{\mbf{c}} \rtimes F^{\mb{Z}/n\mb{Z}}$. 
We note that $\sigma_{\mbf{c}}$ is centralizable in the sense of \cite[\S 5.1]{BrViQuanHopf} by \cite[Proposition 5.3]{BrViQuanHopf} and \cite[Proposition 5.1.7]{KeLyNonsemiTQFT}. 
By \cite[Theorem 5.12, Theorem 6.5]{BrViQuanHopf}, 
$\mc{Z}_{\sigma_{\mbf{c}}} (\Fun_{\mc{G}_{\mbf{c}},\omega} (X_{\mbf{c}}, X_{\mbf{c}}))$ is equivalent to the degree $1$ part of 
$\mc{Z} (\Fun_{\mc{G}_{\mbf{c}} \rtimes \sigma_{\mbf{c}}^{\mb{Z}/n\mb{Z}},\omega} (X_{\mbf{c}}, X_{\mbf{c}}))$. 
By \cite[Example 7.12.19, Corollary 7.16.2]{EGNOTenCat}, 
$\mc{Z} (\Fun_{\mc{G}_{\mbf{c}} \rtimes \sigma_{\mbf{c}}^{\mb{Z}/n\mb{Z}},\omega} (X_{\mbf{c}}, X_{\mbf{c}}))$ is equivalent to 
$\mc{Z} (\mr{Vec}_{\mc{G}_{\mbf{c}} \rtimes \sigma_{\mbf{c}}^{\mb{Z}/n\mb{Z}},\omega})$, where $\mr{Vec}_{\mc{G}_{\mbf{c}} \rtimes \sigma_{\mbf{c}}^{\mb{Z}/n\mb{Z}},\omega}$ is the monoidal category of $(\mc{G}_{\mbf{c}} \rtimes \sigma_{\mbf{c}}^{\mb{Z}/n\mb{Z}})$-graded finite dimensional vector spaces over $\ol{\mb{Q}}_{\ell}$ with the associativity isomorphisms twisted by $\omega$. 
Further $\mc{Z} (\mr{Vec}_{\mc{G}_{\mbf{c}} \rtimes \sigma_{\mbf{c}}^{\mb{Z}/n\mb{Z}},\omega})$ is equivalent to 
$\Sh^{\mc{G}_{\mbf{c}} \rtimes \sigma_{\mbf{c}}^{\mb{Z}/n\mb{Z}}}(\mc{G}_{\mbf{c}} \rtimes F^{\mb{Z}/n\mb{Z}})$ as abelian categories. 
The degree $1$ part of this category is 
$\Sh^{\mc{G}_{\mbf{c}},\sigma_{\mbf{c}}}(\mc{G}_{\mbf{c}})$. 
\end{proof}

\begin{rem}\label{rem:choiceh}
If $h_{\mbf{c}}'$ is another choice of $h_{\mbf{c}}'$, 
we have other automorphicm $\sigma_{\mbf{c}}'$ of $\mc{G}_{\mbf{c}}$ and equivalence 
	\[
\Rep_{\mr{u}}^{\mbf{c}}(H^{\mr{F}}) \cong 
\Sh^{\mc{G}_{\mbf{c}},\sigma_{\mbf{c}}'}(\mc{G}_{\mbf{c}})
\]
as Proposition \ref{prop:unipcatdec}. 
Two equivalences are related by the translation under the image of $h_{\mbf{c}}' h_{\mbf{c}}^{-1}$ in $\mc{G}_{\mbf{c}}$. 
\end{rem}

In the rest of this section, let $G$ be a reductive group over $k$ such that each connected component of $G$ is defined over $k$. 

We take a Borel pair $(B,T)$ of $G$. Let $W$ and $W^{\circ}$ be the Weyl group of $G$ and $G^{\circ}$ with respect to $T$ respectively. 
Let $\Ch (T)$ be the abelian group of isomorphism classes of character sheaves on $T$. We have an isomorphsim 
\begin{equation}\label{eq:chTiso}
 \Ch (T) \cong \Hom (T(k),\ol{\mb{Q}}_{\ell}^{\times}) 
\end{equation}
by \cite[A.3.3 Theorem]{YunRigAutLoc}. 

In the following, when we use results in \cite{LuYuEndHecch}, we refer \cite{LuYuEndHecchv3}, which is a corrected version of \cite{LuYuEndHecch}. 

\begin{defn}\label{defn:ssunippara}
	\begin{enumerate}
		\item\label{en:sspar}
		A semisimple parameter of $G$ is a $W$-orbit $\mf{o} \subset \Ch (T)$ which contains an $F$-stable $W^{\circ}$-orbit.   
		\item\label{en:unippar}
		Let $\mf{o}$ be a semisimple parameter of $G$. 
		A unipotent parameter of $G$ with respect to $\mf{o}$ is a union $\mf{c} \subset W^{\circ} \times \mf{o}$ of a $W$-orbit of $F$-stable two-sided cells $\mf{c}^{\circ} \subset W \times \mf{o}^{\circ}$ in the sense of \cite[11.4]{LuYuEndHecchv3}, 
		where $\mf{o}^{\circ} \subset \mf{o}$ is an $F$-stable $W^{\circ}$-orbit. 
	\end{enumerate}
\end{defn}

For $\mc{L} \in \Ch (T)$, let $W_{\mc{L}}$ and $(W^{\circ})_{\mc{L}}$ be the stabilizer of $\mc{L}$ in $W$ and $W^{\circ}$ respectively. 
We put 
\[
 \Phi_{\mc{L}}=\{ \alpha \in \Phi(G,T) \mid  \textrm{$(\alpha^{\vee})^* \mc{L}$ is trivial}\}, \quad 
  \Phi_{\mc{L}}^+ =\Phi_{\mc{L}} \cap \Phi^+ (G,T). 
\]
Let $(W_{\mc{L}})^{\circ}$ be the Weyl group of $\Phi_{\mc{L}}$. We put 
\[
 \Omega_{\mc{L}}=W_{\mc{L}}/(W_{\mc{L}})^{\circ}, \quad 
 \Omega_{\mc{L}}^{\circ}=(W^{\circ})_{\mc{L}}/(W_{\mc{L}})^{\circ}. 
\]

\begin{lem}\label{lem:ccexist}
	Let $\mf{c}$ be a unipotent parameter of $G$ with respect to a semisimple parameter $\mf{o}$ of $G$. 
	Let $\mc{L}$ be a representative of $\mf{o}$. 
	Then there is a two-sided cell 
	$\mbf{c} \subset (W_{\mc{L}})^{\circ}$ such that $\mf{c} \cap ((W_{\mc{L}})^{\circ} \times \{ \mc{L} \})=\Omega_{\mc{L}}  \mbf{c} \times \{\mc{L}\}$. 
\end{lem}
\begin{proof}
	We put 
	$\mf{o}^{\circ}=W^{\circ}\mc{L}$.  
	We take a two-sided cell $\mf{c}^{\circ} \subset W^{\circ} \times \mf{o}^{\circ}$ as 
	Definition \ref{defn:ssunippara} \ref{en:unippar}. 
	Further, we take a two-sided cell 
	$\mbf{c} \subset (W_{\mc{L}})^{\circ}$ such that $\mf{c}^{\circ} \cap ((W_{\mc{L}})^{\circ} \times \{ \mc{L} \})=\Omega_{\mc{L}}^{\circ} \mbf{c} \times \{\mc{L}\}$ as \cite[11.4]{LuYuEndHecchv3}. 
%Let $W_{\mf{o}^{\circ}}=\mr{Stab}_{W}(\mf{o}^{\circ})$. 
	Then we have 
	\begin{align*}
		\mf{c} \cap ((W_{\mc{L}})^{\circ} \times \{ \mc{L} \})&=(W_{\mf{o}^{\circ}} \mf{c}^{\circ}) \cap ((W_{\mc{L}})^{\circ} \times \{ \mc{L} \}) = (W_{\mc{L}} \mf{c}^{\circ}) \cap ((W_{\mc{L}})^{\circ} \times \{ \mc{L} \})\\ 
		&=W_{\mc{L}} (\mf{c}^{\circ} \cap ((W_{\mc{L}})^{\circ} \times \{ \mc{L} \}))=\Omega_{\mc{L}}  \mbf{c} \times \{\mc{L}\}
	\end{align*}
	by $\mf{c} \cap (W^{\circ} \times \mf{o}^{\circ})= W_{\mf{o}^{\circ}} \mf{c}^{\circ}$ and $W_{\mc{L}}/(W^{\circ})_{\mc{L}} \cong W_{\mf{o}^{\circ}}/W^{\circ}$. 
\end{proof}

We put 
\[
 {}_{\mc{L}}W^{\circ}_{F\mc{L}} =\{ w \in W^{\circ} \mid \mc{L}=wF\mc{L} \}, \quad 
 {}_{\mc{L}}\ul{W}^{\circ}_{F\mc{L}} =(W_{\mc{L}})^{\circ} \backslash {}_{\mc{L}}W^{\circ}_{F\mc{L}}. 
\]
Then each $\beta \in {}_{\mc{L}}\ul{W}^{\circ}_{F\mc{L}}$ contains an element $w^{\beta}$ uniquely characterized by 
$w^{\beta}(\Phi_{\mc{L}}^+) \subset \Phi^+(G,T)$ as \cite[4.2 Lemma]{LuYuEndHecchv3}. 
For a two-sided cell $\mbf{c}$ of $(W_{\mc{L}})^{\circ}$, 
we put 
\[
 \mf{B}_{\mbf{c}}^{\circ}= \{ \beta \in {}_{\mc{L}}\ul{W}^{\circ}_{F\mc{L}} \mid w^{\beta} \sigma \mbf{c} = \mbf{c} \}, \quad 
 \Omega_{\mbf{c}}=\mr{Stab}_{\Omega_{\mc{L}}}(\mbf{c}). 
\]
There is a $\sigma$-twisted conjugation action $\Ad_{\sigma}$ of $\Omega_{\mc{L}}$ on ${}_{\mc{L}}\ul{W}^{\circ}_{F\mc{L}}$ defined by 
\[
 \Ad_{\sigma} (\gamma) (\beta)=\gamma \beta {\sigma}(\gamma)^{-1}, 
\]
because the action of $F$ on $\pi_0 (G)$ is trivial. 
This restricts to an action of 
$\Omega_{\mbf{c}}$ on $\mf{B}_{\mbf{c}}^{\circ}$, 
which is also denoted by $\Ad_{\sigma}$. 
For $\beta \in \mf{B}_{\mbf{c}}^{\circ}$, 
let $\Omega_{\mbf{c},\beta}$ be the stabilize of $\beta$ in $\Omega_{\mbf{c}}$ under $\Ad_{\sigma}$. 

Let $H_{\mc{L}}$ be the connected reductive algebraic group over $\ol{k}$ with a maximal torus identified with $T_{\ol{k}}$ and the root system $\Phi (H,T_{\ol{k}})=\Phi_{\mc{L}}$ as \cite[9.1]{LuYuEndHecchv3}. 
For each $\beta \in {}_{\mc{L}}\ul{W}^{\circ}_{F\mc{L}}$, we define an Frobenius endomorphism $\mr{F}_{\beta}$ on $H_{\mc{L}}$ as in \cite[12.1]{LuYuEndHecchv3}. 
%For each $\beta \in \mf{B}_{\mbf{c}}^{\circ}$, we define a Frobenius structure $\sigma_{\beta}$ on $H_{\mc{L}}$ as \cite[12.1]{LuYuEndHecchv3}. 
Let $\sigma_{\beta}$ be the automorphism of $\wh{H}_{\mc{L}}$ given by the Frobenius structure $\mr{F}_{\beta}$. We have $\sigma_{\beta}=\Ad (\dot{w}^{\beta}) \sigma$. 
We take $h_{\beta} \in \wh{H}_{\mc{L}}$ such that $\Ad (h_{\beta})(\sigma_{\beta} (u_{\mbf{c}}))= u_{\mbf{c}}$, and define an automorphism $\tau_{\beta}$ of $\mc{G}_{\mbf{c}}$ by $\Ad (h_{\beta}) \circ \sigma_{\beta}$. 

Let $\Rep (G(k))$ denote the category of finite dimensional representations of $G(k)$ over $\ol{\mb{Q}}_{\ell}$. 

\begin{thm}\label{thm:RepGcateq}
	We fix a Whittaker datum of $G^{\circ}$. 
	Let $\mf{B}_{\mbf{c},\Omega_{\mbf{c}}}^{\circ}$ be a set of representative of 
	$\mf{B}_{\mbf{c}}^{\circ}/\Ad_{\sigma}(\Omega_{\mbf{c}})$. 
	Then there is a natural equivalence 
	\[
	\Rep (G(k)) \cong \bigoplus_{\mf{o}} \bigoplus_{\mf{c}}
	\bigoplus_{\beta \in \mf{B}_{\mbf{c},\Omega_{\mbf{c}}}^{\circ}} \Sh^{\mc{G}_{\mbf{c}},\tau_{\beta}}(\mc{G}_{\mbf{c}})^{\Omega_{\mbf{c},\beta}}
	\]
	of abelian categories, where $\mf{o}$ runs thorough the semisimple parameters of $G$, $\mf{c}$ runs thorough the unipotent parameters of $G$ with respect to $\mf{o}$ and 
	we take $\mbf{c}$ for each $\mf{c}$ as Lemma \ref{lem:ccexist}. 
\end{thm}
\begin{proof}
	By the decomposition into Lusztig series (\cf \cite[Proposition 11.3.2]{DiMiRepLie2nd}), 
	we have 
	\begin{equation}\label{en:decssp}
		\Rep (G(k)) \cong \bigoplus_{\mf{o}} \Rep_{\mf{o}} (G(k)), 
	\end{equation}
	where $\mf{o}$ runs thorough the semisimple parameters of $G$. 
	
	Let $\mf{o}$ be a semisimple parameter of $G$. 
	We take a representative $\mc{L}$ of $\mf{o}$, and put 
	$\mf{o}^{\circ}=W^{\circ}\mc{L}$.  
	%Then we have 
	%\begin{equation*}
	%\Rep_{\mf{o}} (G(k)) \cong \left( \bigoplus_{\beta \in {}_{\mc{L}} \ul{W}_{F\mc{L}}} \Rep_{\mr{u}}(H^{\mr{F}_{\beta}}) \right)^{\Omega_{\mc{L}}}. 
	%\end{equation*}
	For each unipotent parameter $\mf{c}$ of $G$ with respect to $\mf{o}$, we take $\mbf{c}$ as Lemma \ref{lem:ccexist}.
	%a two-sided cell $\mf{c}^{\circ} \subset W^{\circ} \times \mf{o}^{\circ}$ as Definition \ref{defn:ssunippara} \ref{en:unippar}. Further, we take a two-sided cell $\mbf{c} \subset W_{\mc{L}}^{\circ}$ such that $\mf{c}^{\circ} \cap (W_{\mc{L}}^{\circ} \times \{ \mc{L} \})=\Omega_{\mc{L}}^{\circ} \mbf{c} \times \{\mc{L}\}$ as \cite[11.4]{LuYuEndHecchv3}. 
	
	By \cite[(3.3),(3.4)]{SolEnddisfin}, 
	we have 
	\begin{equation*}
		\Rep_{\mf{o}} (G(k)) \cong \Rep_{\mf{o}} (G^{\circ}(k))^{\pi_0 (G)} \cong  \Rep_{\mf{o}^{\circ}} (G^{\circ}(k))^{\pi_0 (G)_{\mf{o}^{\circ}}} . 
	\end{equation*}
	These induce 
	\begin{align}
		\Rep_{\mf{o}}^{\mf{c}} (G(k)) \cong \Rep_{\mf{o}}^{\mf{c}} (G^{\circ}(k))^{\pi_0 (G)} 
		&\cong 
		\Rep_{\mf{o}^{\circ}}^{W_{\mf{o}^{\circ}} \mf{c}^{\circ}} (G^{\circ}(k))^{\pi_0 (G)_{\mf{o}^{\circ}}} \notag \\ 
		&\cong 
		\Rep_{\mf{o}^{\circ}}^{\mf{c}^{\circ}} (G^{\circ}(k))^{\pi_0 (G)_{\mf{o}^{\circ},\mf{c}^{\circ}}} \label{en:Reppi0oc} 
	\end{align}
	by the construction and $\mf{c} \cap (W^{\circ} \times \mf{o}^{\circ})= W_{\mf{o}^{\circ}} \mf{c}^{\circ}$. 
%There is an induced Frobenius morphism $F$ on $H$ by \cite[Proposition 1.4.28]{GeMaChaFinLie}. 
By \cite[12.7 Corollary]{LuYuEndHecchv3}, 
	\eqref{en:Reppi0oc} is equivalent to 
	\begin{equation*}
		\left( \bigoplus_{\beta \in \mf{B}_{\mbf{c},\Omega_{\mbf{c}}^{\circ}}^{\circ}} \Rep_{\mr{u}}^{\mbf{c}}(H_{\mc{L}}^{\mr{F}_{\beta}})^{\Omega_{\mbf{c},\beta}^{\circ}} \right)^{\pi_0 (G)_{\mf{o}^{\circ},\mf{c}^{\circ}}} 
		\cong 
		\left( \left( \bigoplus_{\beta \in \mf{B}_{\mbf{c}}^{\circ}} \Rep_{\mr{u}}^{\mbf{c}}(H_{\mc{L}}^{\mr{F}_{\beta}}) \right)^{\Omega_{\mbf{c}}^{\circ}} \right)^{\pi_0 (G)_{\mf{o}^{\circ},\mf{c}^{\circ}}} , 
	\end{equation*}
	where $\mf{B}_{\mbf{c},\Omega_{\mbf{c}}^{\circ}}^{\circ}$ is a set of representative of 
	$\mf{B}_{\mbf{c}}^{\circ}/\Ad_{\sigma}(\Omega_{\mbf{c}}^{\circ})$. 
	We note that a Whittaker datum is used in \cite[5.11]{LuYuEndHecchv3} to rigidify IC sheaves, which are used in the construction of the equivalence of categories in \cite[12.7 Corollary]{LuYuEndHecchv3}. 
	This is equivalent to 
	\begin{equation*}
		\left( \bigoplus_{\beta \in \mf{B}_{\mbf{c}}^{\circ}} \Rep_{\mr{u}}^{\mbf{c}}(H_{\mc{L}}^{\mr{F}_{\beta}}) \right)^{\Omega_{\mbf{c}}}  
	\end{equation*}
	using $\pi_0 (G)_{\mf{o}^{\circ},\mf{c}^{\circ}} \cong \Omega_{\mbf{c}}/\Omega_{\mbf{c}}^{\circ}$. 
	Therefore we have  
	\begin{equation*}
		\Rep_{\mf{o}}^{\mf{c}} (G(k)) \cong \bigoplus_{\beta \in \mf{B}_{\mbf{c},\Omega_{\mbf{c}}}^{\circ}} \Rep_{\mr{u}}^{\mbf{c}}(H_{\mc{L}}^{\mr{F}_{\beta}})^{\Omega_{\mbf{c},\beta}} .
	\end{equation*}
	Hence the claim follows from Proposition \ref{prop:unipcatdec}. 
\end{proof}

\begin{rem}
The Whittaker datum in Theorem \ref{thm:RepGcateq} is unique if $Z(G^{\circ})$ is connected by Corollary \ref{cor:whuni}. 
\end{rem}

\begin{lem}\label{lem:edlift}
Let $R$ be a normal, local, Noetherian domain with infinite residue field. 
Let $G$ be a split reductive group scheme over $R$. 
Assume that the characteristic of the residue field is good for $G$. 
Let $s$ be a point of $\Spec R$. 
Let $\mf{n} \in \Lie G_{\ol{k(s)}}$ be a nilpotent element. 
Then there is an equidimensional nilpotent section $X$ of $\Lie G$ such that 
$X_s$ is $G_{\ol{k(s)}}$-conjugate to $\mf{n}$. 
\end{lem}
\begin{proof}
By taking the adjoint quotient, we may assume that $G$ is 
adjoint and simple by \cite[XXIV, Proposition 5.10]{SGA3-3}. If $G$ is not type A, the claim follows from 
\cite[5.4 Theorem]{McNCentnilp} since $G$ is $D$-standard in the sense of \cite[3.9]{McNCentnilp}. In the type A case, we are reduced to the case where $G=\GL_n$. Then the claim follows from \cite[5.4 Theorem]{McNCentnilp} since $\GL_n$ is $D$-standard by \cite[(3.9.3)]{McNCentnilp}. 
\end{proof}

\begin{lem}\label{lem:Aiso}
Let $R$ be a normal, local, Noetherian domain. 
Let $G$ be a reductive group scheme over $R$. 
Assume that the characteristic of the residue field is good for $G$. 
Let $X$ be an equidimensional nilpotent section of $\Lie G$. 
Then we have a natural isomorphism 
\begin{equation*}
 \pi_0( (Z_G(X)/Z(G))_{\ol{s}}) \cong \pi_0( (Z_G(X)/Z(G))_{\ol{t}}) 
\end{equation*}
for any points $s,t$ of $\Spec R$. 
\end{lem}
\begin{proof}
By replacing $R$ by an etale extension, we may assume that $G$ is split. Then we are reduced to the simple adjoint type A case in the same way as the proof of Lemma \ref{lem:edlift} using \cite[7.2 Theorem]{McNCentnilp}. In that case, both groups are trivial. So the proof is complete. 
\end{proof}

For a finite group $\Gamma$, 
let $\Irr_{\ol{\mb{Q}}_{\ell}}(\Gamma)$ 
be the set of isomorphism classes of 
irreducible representations of $\Gamma$ 
over $\ol{\mb{Q}}_{\ell}$.

Let $H$ be a reductive group over a field $\Lambda$. 
Let $u \in H^{\circ}$ be a special unipotent element. 
We put 
\[ 
A_{H^{\circ}}(u)= \pi_0 \left( Z_{H^{\circ}}(u)/
Z( H^{\circ} ) \right). 
\]

\begin{lem}\label{lem:FactAu}
	Let $H$ be a connected reductive group over $\ol{k}$, with Frobenius
	morphism $\mr{F}_H$. 
	Let $C$ be an $F$-stable unipotent conjugacy class of $H$. 
	Then there is $u \in C^{\mr{F}_H}$ such that $\mr{F}_H$ on $A_H(u)$ is trivial. 
\end{lem}
\begin{proof}
	Let $C^{\ad}$ be the image of $C$ in $H^{\ad}$. Then we have $C \cong C^{\ad}$. Hence the claim follows from the adjoint case in \cite[Proposition 13.2.7]{DiMiRepLie2nd}. 
\end{proof}

Assume that the characteristic of $\Lambda$ is good for $H$. 
Let $\ol{A}_{H^{\circ}}(u)$ be the Lusztig's canonical quotient of
$A_{H^{\circ}}(u)$ defined in \cite[13.1]{LusChred}.
We put
\begin{equation}
	\ol{A}_{H}(u) =Z_{H}(u) /\Ker (Z_{H^{\circ}}(u) \to \ol{A}_{H^{\circ}}(u)).
\end{equation}

Let $\chi_{\mc{L}} \colon T(k) \to \ol{\mb{Q}}_{\ell}^{\times}$ be a character corresponding to $\mc{L}$ under \eqref{eq:chTiso}. 
Let $\phi_{\mc{L},0} \colon I_k \to \wh{T}$ be the restriction to $I_k$ of the L-parameter corresponding to $\chi_{\mc{L}}$ under \cite[Proposition 3.8]{IVLpfin}. 
We put $\mc{H}=Z_{\wh{G}}(\phi_{\mc{L},0})$. 

\begin{lem}\label{lem:ZquotOmega}
We have an equivalence  
\begin{equation}\label{eq:ZOmega}
	\Rep (Z_{\ol{A}_{\mc{H}}(u_{\mbf{c}})}(g \ol{h_{\beta} \dot{w}^{\beta} \sigma_q})) \cong \Rep (Z_{\mc{G}_{\mbf{c}}}(g \tau_{\beta}))^{\Omega_{\mbf{c},\beta}} 
\end{equation}
of abelian categories. 
\end{lem}
\begin{proof}
It suffices to show 
\begin{equation*}
	Z_{\ol{A}_{\mc{H}}(u_{\mbf{c}})}(g \ol{h_{\beta} \dot{w}^{\beta} \sigma_q})/ 
	Z_{\mc{G}_{\mbf{c}}}(g \tau_{\beta}) 
	\cong \Omega_{\mbf{c},\beta}. 
\end{equation*}	
We have the isomorphism 
\begin{equation}\label{eq:centconjg}
	Z_{\ol{A}_{\mc{H}}(u_{\mbf{c}})}(g \ol{h_{\beta} \dot{w}^{\beta} \sigma_q})/ 
	Z_{\mc{G}_{\mbf{c}}}(g \tau_{\beta}) 
	\cong Z_{\ol{A}_{\mc{H}}(u_{\mbf{c}})}(\ol{h_{\beta} \dot{w}^{\beta} \sigma_q})/ 
	Z_{\ol{A}_{\mc{H}^{\circ}}(u_{\mbf{c}})}(\ol{h_{\beta} \dot{w}^{\beta} \sigma_q}) 
\end{equation}
given by the conjugation by $g$. 
We have a natural isomorphism 
\begin{equation}\label{eq:AA0Omegac}
	\ol{A}_{\mc{H}}(u_{\mbf{c}}) / \ol{A}_{\mc{H}^{\circ}}(u_{\mbf{c}}) \cong \Omega_{\mbf{c}}. 
\end{equation}
We put 
\begin{equation}
	\widetilde{Z}_{\mc{H}}(u_{\mbf{c}})=\{ (g,\sigma_q^m) \in Z_{{}^{L}G}(u_{\mbf{c}}) \mid 
	\Ad ((g,\sigma_q^m))(\phi_{\mc{L},0} (x))=\phi_{\mc{L},0} (\Ad (\sigma_q^m)(x)) \ \textrm{for all $x \in I_k$}
	\}. 
\end{equation}
Further we put 
\begin{equation}
	\widetilde{A}_{\mc{H}}(u_{\mbf{c}})=\widetilde{Z}_{\mc{H}}(u_{\mbf{c}}) / 
	\Ker (Z_{\mc{H}}(u_{\mbf{c}}) \to \ol{A}_{\mc{H}}(u_{\mbf{c}})). 
\end{equation}
Let $\widetilde{A}_{\mc{H}}(u_{\mbf{c}})_m$ be the subset of $\widetilde{A}_{\mc{H}}(u_{\mbf{c}})$ defined by the condition that the second component is $\sigma_q^m$. 
Then we have a natural bijection 
\begin{equation}\label{eq:A1A0Bc}
	\widetilde{A}_{\mc{H}}(u_{\mbf{c}})_1 / \ol{A}_{\mc{H}^{\circ}}(u_{\mbf{c}}) 
	\cong \mf{B}_{\mbf{c}}^{\circ}.  
\end{equation}
Under the isomorphism \eqref{eq:AA0Omegac} and the bijection \eqref{eq:A1A0Bc},  
the conjugate action of $\ol{A}_{\mc{H}}(u_{\mbf{c}})$ on 
$\widetilde{A}_{\mc{H}}(u_{\mbf{c}})_1 / \ol{A}_{\mc{H}^{\circ}}(u_{\mbf{c}})$ is compatible with the action $\Ad_{\sigma}$ of $\Omega_{\mbf{c}}$ on $\mf{B}_{\mbf{c}}^{\circ}$. 
Therefore we have 
\begin{equation}\label{eq:centOmegacb}
	Z_{\ol{A}_{\mc{H}}(u_{\mbf{c}})}(\ol{h_{\beta} \dot{w}^{\beta} \sigma_q})/ 
	Z_{\ol{A}_{\mc{H}^{\circ}}(u_{\mbf{c}})}(\ol{h_{\beta} \dot{w}^{\beta} \sigma_q}) 
	\cong \Omega_{\mbf{c},\beta}. 
\end{equation}
The claim follows from \eqref{eq:centconjg} and \eqref{eq:centOmegacb}. 
\end{proof}

In the remaining of this section, we assume that $p$ is a good prime. 
Assume that $H$ be a reductive group over $\ol{k}$ with Frobenius morphism $\mr{F}_H$ given by a form of $H$ over $k$. 
Further assume that $u \in (H^{\circ})^{\mr{F}_H}$. The morphism $\mr{F}_H$ naturally acts on $\ol{A}_H(u)$
as an automorphism. We put $\widetilde{A}_{H}(u)=\ol{A}_{H}(u) \rtimes \mr{F}_H^{\mb{Z}}$.

We put
\begin{align*}
	\ol{\mf{M}}&\left( \ol{A}_{H}(u) \subset \widetilde{A}_{H}(u) \right) = \\
	&\{ (x,\rho) \mid x \in \ol{A}_{H}(u) \cdot \mr{F}_H, \ \rho \in \Irr_{\ol{\mb{Q}}_{\ell}} (Z_{\ol{A}_{H}(u)}(x)) \}/{\sim},
\end{align*}
where the equivalence is defined by the conjugacy action of
$\widetilde{A}_{H}(u)$ (\cf~\cite[4.16, 4.21]{LusChred}).
Here we may replace the conjugacy action by that of $\ol{A}_{H}(u)$ because for $x=x_0 \mr{F}_H$ with $x_0 \in \ol{A}_{H}(u)$ we have
$\Ad (\mr{F}_H) (x)=\Ad (x_0^{-1})(x)$ and the actions of $\Ad (\mr{F}_H)$ and $\Ad (x_0^{-1})$ on $Z_{\ol{A}_{H}(u)}(x)$ are same.

Let $G^*$ be the dual group of $G$ over $\ol{k}$.

\begin{defn}
Let $g \in G^*$. Let $s$ and $u$ be the semisimple and unipotent parts of $g$ under the Joradan decomposition.  
\begin{enumerate}
\item 
We say that $g$ is special if $u$ is special in $Z_{G^*}(s)^{\circ}$. 
\item 
If $g$ is special, we put 
\begin{equation*}
 \ol{\mf{M}} \left( \ol{A}_{G^*}(g) \subset \widetilde{A}_{G^*}(g) \right) := 
  \ol{\mf{M}} \left( \ol{A}_{Z_{G^*}(s)}(u) \subset \widetilde{A}_{Z_{G^*}(s)}(u) \right). 
\end{equation*}
\end{enumerate}
\end{defn}

Let $\mr{F}^* \colon G^* \to G^*$ 
be the Frobenius map determined by the rationality of $G$.
We write
\begin{equation}
	T^{*}_0 \subset B_0^{*} 
\end{equation}
for the Borel pair specified by
the pinning in the definition of $G^{*}$.

%We assume that $p$ is a good prime and make the assumption \cite[3.3 Hypothesis]{DiMiLuspar}. 
%Then the derived group of $\wh{G}$ is simply connected by \cite[8.1.8 (ii) and 8.1.12 (6) (b)]{SprLAG}.  

We fix an isomorphism
\begin{equation}\label{en:unityiso}
\mathrm{colim}_{k'} \Hom (k'^{\times},\ol{\mb{Q}}_{\ell}^{\times}) \cong \ol{k}^{\times} , 
\end{equation}
where $k'$ runs through the finite extensions of $k$ in $\ol{k}$ and the transition maps are induced by the norm maps.

\begin{thm}\label{thm:parposdual}
We fix a Whittaker datum of $G$. 
We have a bijection 
\begin{equation*}
	\Irr_{\ol{\mb{Q}}_{\ell}}(G(k)) \cong  \coprod_C  \ol{\mf{M}} \left( \ol{A}_{G^*}(g) \subset \widetilde{A}_{G^*}(g) \right)
\end{equation*}
where $C$ runs over the set of $\mr{F}^*$-stable special conjugacy classes in $G^*$.  
\end{thm}
\begin{proof}
Let $\mf{o}$ be a semisimple parameters of $G$, 
and $\mc{L} \in \mf{o}$. 
Let $s \in T(\ol{k})$ be the element corresponding to $\mc{L}$ under \eqref{eq:chTiso} and \eqref{en:unityiso}.  
We put $\ms{H}=Z_{G^*}(s)$. Then $\ms{H}^{\circ}$ is identified with $\wh{H}_{\mc{L}}$. 
Let $\Phi_s^+$ be the set of the positive roots of $H^{\circ}$ with respect to $T_0^*$ and $H^{\circ} \cap B_0^*$. 
We have a natural bijection 
\begin{equation}\label{eq:connHW}
 \pi_0 (\ms{H}) \cong \{ w \in W(G^*,T_0^*) \mid w(s)=s, \ w(\Phi_s^+)=\Phi_s^+ \} 
\end{equation} 
by \cite[1.1 Proposition]{DiMiLuspar}. 

Let $\mf{c}$ be a unipotent parameters of $G$ with respect to $\mf{o}$. We take $\mbf{c}$ for $\mf{c}$ as Lemma \ref{lem:ccexist}. 
Let $C$ be the unipotent conjugacy classes of $\ms{H}^{\circ}$ corresponding to $\mbf{c}$.

We write $\mr{F}_{0}^* \colon G^* \to G^*$ for 
the Frobenis map determined by the split form of $G^*$ over $k$. 
Then we have $\mr{F}_{\beta}^*=\sigma_\beta \circ \mr{F}_{0}^*$ on $\ms{H}^{\circ}$. 
We take $u_0 \in C^{\mr{F}_{0}^*}$ such that the action of $\mr{F}_{0}^*$ on $A_{\ms{H}^{\circ}}(u_0)$ is trivial. 
By \eqref{eq:connHW}, 
the action of $\mr{F}_{0}^*$ on $A$ is trivial. 
Therefore the action of $\mr{F}_{0}^*$ on $\ol{A}_{\ms{H}}(u_0)$ is trivial since it is naturally isomorphic to $\ol{A}_{\ms{H}^{\circ}}(u_0) \rtimes \mathrm{Stab}_{\pi_0 (\ms{H})}(C)$.

We take $u \in C^{\mr{F}_{\beta}^*}$ and $h \in \ms{H}^{\circ}$ such that 
$\Ad (h)(u_0) =u$. 
Then we have
\begin{equation}
		\Ad (\mr{F}_{0}^*(h))(u_0)=\Ad (\mr{F}_{0}^*(h))(\mr{F}_{0}^*(u_0))
		= \mr{F}_{0}^* (\Ad(h)(u_0))= \mr{F}_{0}^*(u)
\end{equation}
and the commutative diagram
\begin{equation}\label{eq:Au'F0}
	\xymatrix{
		\ol{A}_{\ms{H}}(u)
		\ar[r]^-{\mr{F}_{0}^*} &
		\ol{A}_{\ms{H}}(\mr{F}_{0}^*(u)) \\ 
\ol{A}_{\ms{H}}(u_0). \ar[u]^-{\Ad (h)} \ar[ur]_{\Ad (\mr{F}_{0}^*(h))} & 
	}
\end{equation}
Therefore we have the commutative diagram 
\begin{equation}\label{eq:Fbetasigad}
	\begin{split}
	\xymatrix{
		\ol{A}_{\ms{H}}(u)
		\ar[d]_-{\mr{F}_{0}^*} \ar[r]^-{\sigma_{\beta}} \ar[dr]^-{\mr{F}_{\beta}^*}&
		\ol{A}_{\ms{H}}(\sigma_{\beta}(u)) \ar[d]^-{\Ad(\sigma_{\beta}(\mr{F}_{0}^*(h)h^{-1}))} \\ 
		\ol{A}_{\ms{H}}(\mr{F}_{0}^*(u)) \ar[r]^-{\sigma_{\beta}} & \ol{A}_{\ms{H}}(u). 
	}
\end{split}
\end{equation}
%Hence $\mr{F}_{\beta}^*$ on $\ol{A}_{\ms{H}}(u)$ is equal to $\Ad(\sigma_{\beta}(\mr{F}_{0}^*(h)h^{-1})) \circ \sigma_{\beta}$. 

We write $W(\ol{k})$ for the ring of Witt vectors over $\ol{k}$. 
%\cite[9.3.3]{SprLAG}. 
Let $\msc{H}$ be the split reductive group over $W(\ol{k})$ lifting $\ms{H}^{\circ}$. 
Let $\mf{n} \in \Lie (\ms{H})$ 
be the nilpotent element 
corresponding to 
$u$ under the Springer isomorphism. 
By Lemma \ref{lem:edlift}, 
we can take an equidimensional nilpotent section 
$\wh{\mf{n}}$
of 
$\Lie \msc{H}$ over $W(\ol{k})$
lifting $\mf{n}$. 

We show that 
$\mr{Trans}_{\msc{H}}(\sigma_{\beta}(\wh{\mf{n}}), \wh{\mf{n}})$
is smooth, 
where $\mr{Trans}$ is defined as \cite[2.3]{McNCentnilp}. 
Since $\mr{Trans}_{\msc{H}}(\sigma_{\beta}(\wh{\mf{n}}), \wh{\mf{n}})$ is a torsor under 
$Z_{\msc{H}}(\wh{\mf{n}})$, the smoothness follows from 
\cite[5.2 Proposition]{McNCentnilp}. 

By the smoothness of $\mr{Trans}_{\msc{H}}(\sigma_{\beta}(\wh{\mf{n}}), \wh{\mf{n}}) $ over $W(\ol{k})$, we may lift 
\begin{equation*}
\sigma_{\beta}(\mr{F}_{0}^*(h)h^{-1}) \in \mr{Trans}_{\msc{H}}(\sigma_{\beta}(\wh{\mf{n}}), \wh{\mf{n}})(\ol{k})
\end{equation*}
to 
\begin{equation*}
\wh{h}_{\beta} \in \mr{Trans}_{\msc{H}}(\sigma_{\beta}(\wh{\mf{n}}), \wh{\mf{n}})(W(\ol{k})). 
\end{equation*}
Taking an embedding $W(\ol{k}) \hookrightarrow \ol{\mb{Q}}_{\ell}$. 
Let $u_{\mbf{c}} \in \wh{H}_{\mc{L}}$ be the unipotent element corresponding under the Springer isomorphism to the image of $\wh{\mf{n}}$ in $\Lie \wh{H}_{\mc{L}}$ by the embedding $W(\ol{k}) \hookrightarrow \ol{\mb{Q}}_{\ell}$. 
By Lemma \ref{lem:Aiso}, we have 
\begin{equation}\label{eq:A0ucu}
 \ol{A}_{\mc{H}^{\circ}} (u_{\mbf{c}}) \cong 
 \ol{A}_{\ms{H}^{\circ}} (u). 
\end{equation}
By \eqref{eq:AA0Omegac}, \eqref{eq:A0ucu} and 
\begin{equation*}
 \ol{A}_{\ms{H}}(u) / \ol{A}_{\ms{H}^{\circ}}(u) \cong \mathrm{Stab}_{\pi_0 (\ms{H})}(C) \cong \Omega_{\mbf{c}} ,  
\end{equation*}
we have 
\begin{equation}\label{eq:Aucu}
	\ol{A}_{\mc{H}} (u_{\mbf{c}}) \cong 
	\ol{A}_{\ms{H}} (u). 
\end{equation}

Let $h_{\beta} \in \wh{H}_{\mc{L}}$ be the image of $\wh{h}_{\beta}$. 
Using this $h_{\beta}$, we define $\tau_{\beta}$ in Theorem \ref{thm:RepGcateq}. 
Then the conjugate action of $\ol{h_{\beta} \dot{w}^{\beta} \sigma_q}$ on $\ol{A}_{\mc{H}} (u_{\mbf{c}})$ is 
is compatible with 
the action of $F_{\beta}^*$ on $\ol{A}_{\ms{H}}(u)$ 
under \eqref{eq:Aucu} 
by \eqref{eq:Fbetasigad} and the construction of $h_{\beta}$. 
Therefore we have 
\begin{equation}\label{eq:ZOmega}
		Z_{\ol{A}_{\mc{H}}(u_{\mbf{c}})}(g \ol{h_{\beta} \dot{w}^{\beta} \sigma_q})  \cong 
	Z_{\ol{A}_{\ms{H}}(u)}(g F_{\beta}^*).  
\end{equation}
The claim follows from Lemma \ref{lem:ZquotOmega} and \eqref{eq:ZOmega}. 
\end{proof}

\section{Langlands correspondence for finite fields}

Let $G$ be a reductive algebraic group over $k$. 
The aim of this section is constructing the Langlands correspondence for $G$, which is a natural map 
\begin{equation}\label{eq:LLCfin}
\mc{L}_G \colon \Irr_{\ol{\mb{Q}}_{\ell}}(G(k)) \to 
\Phi_{\ol{\mb{Q}}_{\ell}}(G)_{\mr{sp}}. 
\end{equation}
This is correspondence concerning L-parameters of Weil--Deligne type explained in Section \ref{sec:WD}. 
However, for a technical reason, we need to use auxiliary L-parameters of $\SL_2$-type, which we explain below, in the course of our construction of \eqref{eq:LLCfin}. 

Let $\Lambda$ be a field of characteristic zero, 
and ${}^L G$ the $L$-group of $G$ over $\Lambda$. 

\begin{defn}
An L-parameter of $\SL_2$-type is a morphism 
      \begin{equation}
        \psi \colon \SL_2 \times W_k \rightarrow {}^L G 
      \end{equation} 
      of group schemes over $\Lambda$ which is 
compatible with the projections to $\Gal
        (\ol{k}/k)$. 

We say that the L-parameter $\psi$ is special 
if $\psi|_{\mb{G}_a(K)}(1)$ is a special unipotent element of 
$\wh{G}^{\psi(I_k)}$. 
We say that $\psi$ is Frobenius semisimple if
$\psi(\sigma_q)$ is semisimple in ${}^L G$. 
We say that $\psi$ is unipotent if
$\psi(I_k)$ is trivial. 
\end{defn}

We put $\psi_0=\psi|_{\SL_2 \times I_k}$ and 
\begin{equation}
A(\psi_0)=\pi_0 \left( Z_{\wh{G}}(\psi_0)/Z(\wh{G}) \right). 
\end{equation}
For $\varphi_0$ defined by 
\begin{eqnarray*}
 \varphi_0 (a ,w) =\psi_0 \left( 
 \begin{pmatrix}
 	1 & a \\ 0 & 1 
 \end{pmatrix}, w \right)
\end{eqnarray*}
for $(a,w) \in \Ga \times I_k$, 
we have a natural isomorphism 
$A (\psi_0) \cong A (\varphi_0)$ since 
$Z_{\wh{G}}(\varphi_0)$ is a semidirect product of 
$Z_{\wh{G}}(\psi_0)$ with a connected unipotent group by \cite[Proposition 3.3]{BMIYJMmor}. 
Then we define $\ol{A}(\psi_0)$ as the quotient of $A(\psi_0)$ corresponding to 
Lusztig's canonical quotient $\ol{A}(\varphi_0)$ of $A(\varphi_0)$. 

We put 
\begin{equation}
 \widetilde{Z}(\psi_0)=\{ (g,\sigma_q^m) \in {}^{L}G \mid 
 \Ad ((g,\sigma_q^m))(\psi_0 (x))=\psi_0 (\Ad (\sigma_q^m)(x)) \ \textrm{for all $x \in \SL_2 \times I_k$}
 \}. 
\end{equation}
Further we put 
\begin{equation}
 \widetilde{A}(\psi_0)=\widetilde{Z}(\psi_0) / 
 \Ker (Z_{\wh{G}}(\psi_0) \to \ol{A}(\psi_0)). 
\end{equation}

We have 
$\psi (\sigma_q) \in \widetilde{Z}(\psi_0)$. 
Let $\ol{\psi (\sigma_q)}$ be the image of 
$\psi (\sigma_q)$ under the natural projection 
\[
 \widetilde{Z}(\psi_0) \to 
 \widetilde{A}(\psi_0). 
\]
We say that two 
L-parameters $\psi$ and $\psi'$ are equivalent 
if the following condition is satisfied: 
there is $g \in \wh{G}$ such that 
$\Ad (g)(\psi_0)=\psi_0'$ and 
$\ol{\psi (\sigma_q)}$ corresponds to 
$\ol{\psi' (\sigma_q)}$ under the bijection 
\[
 \widetilde{A}(\psi_0) \cong 
 \widetilde{A}(\psi'_0) 
\]
induced by $\Ad (g)$, where 
$\psi_0=\psi|_{\SL_2 \times I_k}$ and 
$\psi_0'=\psi'|_{\SL_2 \times I_k}$. 
Let $\Psi_{\Lambda}(G)$ be the equivalence classes of 
Frobenius semisimple L-parameters over $\Lambda$ of $G$. 
We write $\Psi_{\Lambda}(G)_{\mr{sp}} \subset \Psi_{\Lambda}(G)$ for 
the equivalence class of special ones. 

We put 
\begin{equation}
A_{\psi} = Z_{\ol{A}(\psi_0)} (\ol{\psi (\sigma_q)}) . 
\end{equation} 
We assume and fix $q^{1/2} \in \Lambda$. 

\begin{prop}\label{prop:bijSLWD}
Sending an L-parameter 
$\psi$ of $\SL_2$-type to $\varphi$ defined by 
\begin{equation}
 \varphi (a,w)=\psi \left( 
\begin{pmatrix}
1 & a \\ 0 & 1 
\end{pmatrix}
\begin{pmatrix}
\lvert w \rvert^{1/2} & 0 \\ 0 & \lvert w \rvert^{-1/2} 
\end{pmatrix}, w \right) , 
\end{equation}
we have bijections 
$\Psi_{\Lambda}(G) \cong \Phi_{\Lambda}(G)$ 
and 
$\Psi_{\Lambda}(G)_{\mr{sp}} \cong \Phi_{\Lambda}(G)_{\mr{sp}}$. 
Further we have a natural isomorphism 
$A_{\psi} \cong A_{\varphi}$ for $\psi$ and $\varphi$ as above. 
\end{prop}
\begin{proof}
Sending $\psi$ to $\varphi$, 
we have a bijection between 
the $\wh{G}$-equivalence classes of 
the Frobenius semisimple L-parameters 
of $\SL_2$-type 
and 
the $\wh{G}$-equivalence classes of 
the Frobenius semisimple Weil--Deligne L-parameters 
by \cite[Theorem 6.16]{BMIYJMmor} and \cite[Proposition 1.7]{ImaLLCell}. 
We have a natural isomorphism 
$\ol{A} (\psi_0) \cong \ol{A} (\varphi_0)$ by the definition of $\ol{A} (\psi_0)$. 
Therefore 
\begin{equation}\label{eq:wtZhom}
\widetilde{Z}(\psi_0) \to 
\widetilde{Z}(\varphi_0) ; (g,\sigma_q^m) \mapsto 
\psi_0 \left( 
\begin{pmatrix}
	\lvert \sigma_q \rvert^{m/2} & 0 \\ 0 & \lvert \sigma_q \rvert^{-m/2} 
\end{pmatrix}, 1 \right)(g,\sigma_q^m)
\end{equation}
induces the isomorphism 
$\widetilde{A} (\psi_0) \cong \widetilde{A} (\varphi_0)$, 
where \eqref{eq:wtZhom} is a group homomorphism because any element of $Z_{\wh{G}}(\psi_0)$ commutes with $\psi \left( 
\begin{pmatrix}
	\lvert \sigma_q \rvert^{m/2} & 0 \\ 0 & \lvert \sigma_q \rvert^{-m/2} 
\end{pmatrix}, 1 \right)$. 
Therefore we obtain the claims. 
\end{proof}

\begin{thm}\label{thm:LCfin}
We fix a Whittaker datum of $G$. 
Then we have a natural map 
	\[
	\mc{L}_G \colon \Irr_{\ol{\mb{Q}}_{\ell}}(G(k)) \to 
	\Phi_{\ol{\mb{Q}}_{\ell}}(G)_{\mr{sp}},  
	\]
and a natural bijection between $\mc{L}_G^{-1}(\varphi)$ and 
$\Irr_{\ol{\mb{Q}}_{\ell}}(A_{\varphi})$ for $\varphi \in \Phi_{\ol{\mb{Q}}_{\ell}}(G)_{\mr{sp}}$. 
\end{thm}
\begin{proof}
We construct a map 
	\begin{equation}
	\mc{L}_G^{\Psi} \colon \Irr_{\ol{\mb{Q}}_{\ell}}(G(k))  \to 
		\Psi_{\ol{\mb{Q}}_{\ell}}(G)_{\mr{sp}}. 
	\end{equation}
Let $\pi \in \Irr_{\ol{\mb{Q}}_{\ell}}(G(k))$. 
Let $\mc{F}_{\pi}$ be an irreducible object of 
$\Sh^{\mc{G}_{\mbf{c}},\tau_{\beta}}(\mc{G}_{\mbf{c}})^{\Omega_{\mbf{c},\beta}}$ corresponding to $\pi$ under 
Theorem \ref{thm:RepGcateq} 
for some $\mf{o}$, $\mf{c}$ and 
$\beta \in \mf{B}_{\mbf{c},\Omega_{\mbf{c}}}^{\circ}$. 
This gives an element $g \in \mc{G}_{\mbf{c}}$ and a representaion of $Z_{\mc{G}_{\mbf{c}}}(g\sigma_{\beta})$ with $\Omega_{\mbf{c},\beta}$-equivariant structure. 

Let $\mc{L}$ be a representative of $\mf{o}$. 
Recall that $\mc{H}=Z_{\wh{G}}(\phi_{\mc{L},0})$. 
Let $u_{\mbf{c}} \in \mc{H}^{\circ}$ an element of the unipotent conjugacy classs corresponding to $\mbf{c}$. 
We have an identification 
$\ol{A}_{\mc{H}^{\circ}}(u_{\mbf{c}}) \cong \mc{G}_{\mbf{c}}$. 
Let $\wt{g} \in Z_{\mc{H}^{\circ}}(u_{\mbf{c}})$ be a lift of $g$. 

We define $\psi$ by 
\[	 
 \psi|_{I_k}=\phi_{\mc{L},0}, \quad 
\psi \begin{pmatrix}
	1 & 1 \\ 0 & 1 
\end{pmatrix}=u_{\mbf{c}}, \quad  
\psi (\sigma_q) = \wt{g} h_{\beta} \dot{w}^{\beta} \rtimes \sigma_q. 
\]
	We note that 
\begin{align*}
	\Ad(\psi (\sigma_q))(\phi_{\mc{L},0})=\phi_{\mc{L},0}^q. 
\end{align*}
We also have 
$\Ad (\psi (\sigma_q))(u_{\mbf{c}})=u_{\mbf{c}}$. 
We put $\mc{L}_G^{\Psi} (\pi)=\psi$. 

We have natural bijections 
\begin{equation}\label{eq:Apicent}
	A_{\psi}\cong Z_{\ol{A}_{\mc{H}}(u_{\mbf{c}})}(\ol{\psi(\sigma_q)}) \cong Z_{\ol{A}_{\mc{H}}(u_{\mbf{c}})}(g \ol{h_{\beta} \dot{w}^{\beta} \sigma_q}).
\end{equation}
By \eqref{eq:Apicent} and Lemma \ref{lem:ZquotOmega}, 
we have 
\begin{equation}\label{eq:ApsiZOmega}
\Rep (A_{\psi}) \cong \Rep (Z_{\mc{G}_{\mbf{c}}}(g \tau_{\beta}))^{\Omega_{\mbf{c},\beta}} . 
\end{equation}
By the construction and \eqref{eq:ApsiZOmega}, the fiber $(\mc{L}_G^{\Psi})^{-1}(\psi)$ 
is parametrized by 
the isomorphism classes of the irreducible representations of $A_{\psi}$. 

We define $\mc{L}_G$ 
as the composite of $\mc{L}_G^{\Psi}$ 
and 
the natural bijection 
$\Psi_{\ol{\mb{Q}}_{\ell}}(G)_{\mr{sp}} \cong 
\Phi_{\ol{\mb{Q}}_{\ell}}(G)_{\mr{sp}}$ in Proposition \ref{prop:bijSLWD}. 
\end{proof}

\section{Relation with the categorical local Langlands}\label{sec:RelcatLL}

We discuss relation with the categorical local Langlands formulated in 
\cite{BCHNCohSpr}, 
\cite{FaScGeomLLC}, \cite{HelderIH}, \cite{ZhuCohLp}. 
Actually, we need only the tame part of the categorical local Langlands, which is constructed in \cite{ZhuTamecatLLC}. 
%In this section, we assume Conjecture \ref{conj:FLC}. 

Let $F$ be a non-archimedean local field with the residue field $k$. 
Let $\mc{G}$ be a reductive group scheme over $\mc{O}_F$ 
whose special fiber is $G$. 
Let $\mbf{G}$ be the generic fiber of $\mc{G}$. 
We note that $\wh{G}$ is the dual group of $\mbf{G}$. 
Let $I_F$ and $P_F$ be the inertia subgroup and wild inertia subgroup of the Weil group $W_F$ of $F$. 

Let $\mc{B}$ be a Borel subgroup of $\mc{G}$. 
Let $\mbf{B} \subset \mbf{G}$ and $B \subset G$ be the Borel subgroups corresponding to $\mc{B}$. 
Let $\psi \colon R_{\mr{u}}(\mbf{B})(F) \to \ol{\mb{Q}}_{\ell}^{\times}$ be a generic character such that 
the restriction of $\psi$ to $R_{\mr{u}}(\mc{B})(\mc{O}_F)$ induces a generic character $\ol{\psi}$ of $R_{\mr{u}}(B)(k)$. 

Let $\pi$ be an irreducible representation of $G(k)$ over $\ol{\mb{Q}}_{\ell}$. 
Let $[\varphi]$ be the equivalence class of L-parameters corresponding to $\pi$. 
Let $\widetilde{\sigma}_q \in W_F$ be a lift of $\sigma_q$. 
Then we have   
\begin{equation}\label{eq:WPWkisom}
W_F/P_F \cong (I_F/P_F) \rtimes \mb{Z} \cong W_k, 
\end{equation}
where the first isomorphism is given by $\widetilde{\sigma}_q$. 
Let 
\begin{equation*}
p_{\widetilde{\sigma}_q} \colon \mathit{WD}_F \to \mathit{WD}_k
\end{equation*}
be the natural morphism given by \eqref{eq:WPWkisom}. 

We say that a tame L-parameter $\widetilde{\varphi}$ for $\mbf{G}$ is a lift of $[\varphi]$ if 
$\widetilde{\varphi}$ is given by 
a representative of $[\varphi]$ and $p_{\widetilde{\sigma}_q} \colon \mathit{WD}_F \to \mathit{WD}_k$ for some lift $\widetilde{\sigma}_q$ of $\sigma_q$. 

Let $Z^1(W_F,\wh{G})$ be the moduli of L-parameters over $\ol{\mb{Q}}_{\ell}$ for $\mbf{G}$. 
Let $C_{[\varphi]}$ be the connected components which contain a lift of $[\varphi]$. 
Let 
\begin{equation*}
u_G \colon Z^1(W_F,\wh{G}) \to 
\mc{N}_{\wh{G}}
\end{equation*} 
be the unipotent monodromy morphism in \cite[VIII.2.1]{FaScGeomLLC}, 
where $\mc{N}_{\wh{G}}$ is the nilpotent cone of $\Lie (\wh{G})$.  
Let $N_{[\varphi]}$ be the image of the set of the lifts of $[\varphi]$. 
We write $\ol{N}_{[\varphi]}$ for the closure of $N_{[\varphi]}$ in $\mc{N}_{\wh{G}}$. 
We define $X_{\leq [\varphi]}$ as 
$C_{[\varphi]} \cap u_G^{-1} (\ol{N}_{[\varphi]})$.

Let $\mc{S}_{\pi}$ be the coherent sheaf on the moduli of L-parameters given by 
$\cInd_{\mc{G}(\mc{O}_F)}^{\mbf{G}(F)} \pi$ 
under the tame categorical local Langlands correspondence.

Let $\widetilde{\varphi}$ be a lift of $[\varphi]$ with respect to $\widetilde{\sigma}_q$. 
We define 
\begin{equation*}
h_{\widetilde{\varphi}} \colon \wh{G} \times Z_{\wh{G}} (\varphi_0) \to 
Z^1(W_F,\wh{G})
\end{equation*}
by 
$h_{\widetilde{\varphi}}(g,g')|_{I_F} =\Ad (g) (\varphi_0)$ 
and 
$h_{\widetilde{\varphi}}(g,g')(\widetilde{\sigma}_q)=\Ad (g)(\widetilde{\varphi}(\widetilde{\sigma}_q) g')$. 
Let $X_{[\varphi]}$ be the image of 
$h_{\widetilde{\varphi}}$. 
We note that $X_{[\varphi]}$ is independent of the choice of the lift $\widetilde{\varphi}$. 
Then $h_{\widetilde{\varphi}}$ is a 
$Z_{\wh{G}} (\varphi_0)$-torsor over 
$X_{[\varphi]}$ as in \cite[\S 2.3]{DHKMModLp}, 
where the action of $Z_{\wh{G}} (\varphi_0)$ on 
$\wh{G} \times Z_{\wh{G}} (\varphi_0)$ is given by 
\begin{equation*}
 g'' \cdot (g,g')=(gg''^{-1} , \widetilde{\varphi}(\widetilde{\sigma}_q)^{-1} g'' \widetilde{\varphi}(\widetilde{\sigma}_q) g' g''^{-1}). 
\end{equation*}

Let $\rho_{\pi}$ be the representation of $A_{\varphi}$ corresponding to $\pi$ by Theorem \ref{thm:LCfin} with respect to $\ol{\psi}$. 
We define $\widetilde{\rho}_{\pi}$ as the inflation of 
$\Ind_{A_{\varphi}}^{\ol{A}(\varphi_0)} \rho_{\pi}$ 
under $Z_{\wh{G}} (\varphi_0) \to \ol{A}(\varphi_0)$. 
Let $V(\widetilde{\rho}_{\pi})$ be the vector bundle on  $X_{[\varphi]}$ given by $\widetilde{\rho}_{\pi}$ and the $Z_{\wh{G}} (\varphi_0)$-torsor $h_{\widetilde{\varphi}}$. 
%Since the projection $\wh{G} \times Z_{\wh{G}} (\varphi_0) \to \wh{G}$ is $Z_{\wh{G}} (\varphi_0)$-equivariant, $\widetilde{\rho}_{\pi}$ gives a $Z_{\wh{G}} (\varphi_0)$-equivariant vector bundle on $\wh{G} \times Z_{\wh{G}} (\varphi_0)$. 
%Further this descends to a vector bundle $V(\widetilde{\rho}_{\pi})$ on $X_{(\widetilde{\varphi})}$. 

\begin{conj}\label{conj:catfin}
The support of $\mc{S}_{\pi}$ is contained in $X_{\leq [\varphi]}$. 
The restriction of $\mc{S}_{\pi}$ to $X_{[\varphi]}$ is isomorphic to 
$V(\widetilde{\rho}_{\pi})$. 
\end{conj}

\begin{rem}
If $\varphi|_{\Ga}$ is trivial, $X_{\leq [\varphi]}=X_{[\varphi]}$. %(\cf \cite[Theorem 5.16]{BMIYJMmor}). 
Hence Conjecture \ref{conj:catfin} gives a full description of $\mc{S}_{\pi}$ in this case. 
\end{rem}

\begin{rem}
When $\pi$ is principal representation, $S_{\pi}$ should appear
as a direct summand of the Springer coherent sheaf studied in
\cite{BCHNCohSpr}, \cite{HelderIH} and \cite{ZhuCohLp}.
If $G$ is $\GL_2$ or $\PGL_2$ and $\pi$ is a unipotent principal representation such a decomposition is known by
\cite[Proposition 4.27]{HelderIH} and \cite[Example 4.4.4]{ZhuCohLp}.
In general, $\mc{S}_{\pi}$ is not locally free on $X_{\leq [\varphi]}$ by \cite[Remark 4.28]{HelderIH}.
\end{rem}

%\bibliographystyle{test2}
%\bibliography{reference}

\noindent
Naoki Imai\\
Graduate School of Mathematical Sciences, The University of Tokyo, 
3-8-1 Komaba, Meguro-ku, Tokyo, 153-8914, Japan \\
naoki@ms.u-tokyo.ac.jp\\

\end{document}